\documentclass{amsart}
\usepackage{hyperref}
\usepackage{graphicx}
\usepackage{color}
\usepackage{url}

\usepackage{amssymb}
\usepackage{mathrsfs}

\theoremstyle{plain}
\newtheorem{theorem}{Theorem}[section]
\newtheorem{corollary}[theorem]{Corollary}
\newtheorem{lemma}[theorem]{Lemma}
\newtheorem{proposition}[theorem]{Proposition}

\theoremstyle{definition}
\newtheorem{definition}[theorem]{Definition}

\theoremstyle{remark}
\newtheorem{example}{Example}[section]

\numberwithin{equation}{section}
\numberwithin{figure}{section}

\newcommand{\cG}{\mathcal{G}}

\newcommand{\cR}{{\mathcal R}}
\newcommand{\cO}{{\mathcal O}}

\newcommand{\CC}{{\mathbb C}}
\newcommand{\HH}{{\mathbb H}}
\newcommand{\PP}{{\mathbb P}}
\newcommand{\RR}{{\mathbb R}}
\newcommand{\ZZ}{{\mathbb Z}}

\renewcommand{\a}{\alpha}
\renewcommand{\b}{\beta}
\renewcommand{\c}{\gamma}

\renewcommand{\d}{\partial}

\newcommand{\<}{\langle}
\renewcommand{\>}{\rangle}

\begin{document}

\title[JNR monopoles]{JNR monopoles}
\author{Michael K. Murray}
\address[Michael K. Murray]
{School of Mathematical Sciences\\
University of Adelaide\\
Adelaide, SA 5005 \\
Australia}
\email[Michael K. Murray]{michael.murray@adelaide.edu.au}
\author{Paul Norbury}
\address[Paul Norbury]
{School of Mathematics and Statistics \\
University of Melbourne\\
Vic 3010\\
Australia}
\email[Paul Norbury]{norbury@unimelb.edu.au}

\thanks{MKM acknowledges support under the  Australian
Research Council's {\sl Discovery Projects} funding scheme (project number DP180100383).  PN acknowledges support under the  Australian
Research Council's {\sl Discovery Projects} funding scheme (project numbers DP170102028 and DP180103891).
}

\subjclass[2010]{81T13, 53C07, 14D21}
%


\begin{abstract}
We review the theory of  JNR, mass $\frac{1}{2}$ hyperbolic monopoles in particular
their spectral curves and rational maps.  These are used to establish conditions 
for a spectral curve to be the spectral curve of a  JNR monopole and to show that that  rational map of a JNR monopole 
monopole
arises by scattering using results of Atiyah. We show that for JNR monopoles the 
holomorphic sphere has a remarkably simple form and show that this can be used
to give a formula for the energy density at infinity.  In conclusion we illustrate
some examples of the energy-density at infinity of JNR monopoles.
\end{abstract}

\maketitle

\dedicatory{
\begin{center}
Dedicated to the memory of Sir Michael Atiyah
\end{center} 
}


\tableofcontents

\section{Introduction}

The first author was introduced to monopoles by Sir Michael Atiyah in 1980 when  Atiyah proposed
that generalising Nigel Hitchin's then recent article \cite{Hit} from $SU(2)$ to an arbitrary 
compact Lie group would make a good DPhil project.  After some years of expert supervision by Atiyah
this proved to be the case.  

In 1989  Atiyah came to the Australian National University to attend
 a Miniconference on Geometry and Physics \cite{Ati89}.  He talked to Rodney Baxter about his work with Perk and Au-Yang 
 on the chiral Potts model.  His interest was particularly in the fact that their new solutions of the chiral Potts
 model \cite{BaxPerAuY} involved algebraic curves of genus greater than  $1$.  This led to him making what  he later
 told the second author was a `flight of fancy' that this 
 class of curves might be related to the spectral curves of monopoles, another interesting class of curves and that
 the chiral Potts integrable model solutions might also be related  to monopoles and instantons and the
 whole question of self-duality and integrability \cite{MasWoo}. 
 
  In 1990 the first author visited
 Atiyah at Oxford and we resolved the first part of this conjecture identifying the Baxter, Perk, Au-Yang curves
 with particular spectral curves of zero mass hyperbolic monopoles \cite{Ati91, AtiMur}. Hyperbolic monopoles
 had been introduced earlier by Atiyah in \cite{Ati}. However finding a relationship 
 between monopoles and the actual solutions of the chiral Potts model has remained elusive. We discussed 
 it a number of time over the years in various locations including the Master's Lodge at Trinity College and at various times
 in Edinburgh. Part of the difficulty 
 would seem to be  understanding the significance of the mass zero hyperbolic monopoles which are probably better thought of 
 as limits of positive mass hyperbolic monopoles\footnote{See for example \cite{MurNorSin}
 for further discussion of this point.}.
  
Trying  to resolve Atiyah's  conjecture has motivated a number of interesting explorations of hyperbolic monopoles by the 
 current two authors and their colleagues \cite{MurSin, MurNorSin, NorRom, MurSin00}.  While it would 
 be particularly nice for this volume  to say we were about to present the resolution of this conjecture
 that is sadly not the case.  We present instead a discussion of some interesting results involving
 mass $\frac{1}{2}$ hyperbolic monopoles.  It seemed appropriate however in the current circumstances to explain the deeper underlying motivation 
 for our interest in hyperbolic monopoles.

Monopoles are solutions of the  Bogomolny equations which arise as the dimensional reduction of time-invariant instantons on $\RR^4$. 
In \cite{Ati} Atiyah noticed that by exploiting the conformal invariance of the self-duality equations 
one can pass from instantons on $S^4 - S^2$ to instantons on $\HH^3 \times S^1$ with metric the product
of the hyperbolic metric and the circle metric and then by demanding rotation invariance in the $S^1$ direction 
to solutions of the Bogomolny equations on hyperbolic space $\HH^3$.  As in the Euclidean case
appropriate boundary conditions can be imposed resulting in the definition of the mass and charge 
of the monopole.  One finds that the mass is in $\frac{1}{2} \ZZ$
if and only if under the conformal transformation above the instanton on $S^4 - S^2$ extends to an instanton 
on $S^4$. In \cite{ManSut} Manton and Sutcliffe showed how to 
construct particular mass  $\frac{1}{2}$ monopoles known as JNR monopoles in this way using  an ansatze for instantons due
to Jackiw, Nohl and Rebbi \cite{JacNohReb},
the formula for which also occurred in \cite{CorFai}.  These instantons were early examples of 
what later became known as the Atiyah, Drinfeld, Hitchin, Manin (ADHM) construction \cite{AtiHitDri}. In \cite{ManSut} Manton and Sutcliffe construct JNR monopoles
with various symmetries including: spherical, axial, tetrahedral, octahedral and icosohedral.  
Following on this work \cite{BolCocSut, Coc} give formulae for the spectral curve and the rational 
map of the general JNR monopole and for these particular symmetric cases.  What is quite remarkable about JNR monopoles
is how simple and explicit the formulae are for the spectral curves and rational maps, not just in the symmetric cases but 
also for the general JNR monopole. 

Using these results in Section \ref{sec:grid} we  make a number of additions to this theory.  We give a general geometric constraint that characterises the spectral curves of JNR
monopoles. This is the existence of a so-called grid.  We show that this can be used to define the section of the line bundle $L^{2p+N}$  the existence of which is one of the
important properties of
a spectral curve of a mass $p$ charge $N$ monopole satisfies.  In turn we show that this section can be used as scattering data to define the rational map 
as in the Euclidean case.  

We represent the monopole in terms of holomorphic, or twistor, data rather than writing down the explicit fields, for most of the paper.  The  energy density of a monopole is the $L^2$ norm of its curvature which gives a type of location of the fields of the monopole.   Moreover, a hyperbolic monopole is uniquely determined by the limit at infinity of its energy density.  In Section~\ref{sec:energy} we determine the energy density of the JNR monopole at infinity.  To achieve this we use the  holomorphic sphere of the  monopole, described in Section~\ref{sec:holsph}, which was constructed for hyperbolic monopoles of any mass  in work of the current authors with Singer \cite{MurNorSin}. This is a 
holomorphic map $q \colon \PP_1 \to \PP_N$ determined up to the action of $U(N+1)$. One satisfying aspect of this paper is a simple formula for the holomorphic sphere 
in terms of the JNR data given by Proposition \ref{prop:holo-sphere}, which leads to a formula for the energy density of the monopole at infinity.   In Section~\ref{sec:sphrat} we discuss the relationship between the holomorphic sphere and the different rational maps obtained by scattering from different points at infinity.

\section{Hyperbolic monopoles}

A monopole in Euclidean, respectively hyperbolic, space is a pair $(A,\Phi)$ consisting of a connection $A$ with
$L^2$ curvature $F_A$ defined on a trivial bundle $E$ over $\RR^3$, respectively $\HH^3$,
with structure group $SU(2)$, and a Higgs field
$\Phi:\RR^3 (\HH^3) \rightarrow{\bf su}(2)$ that solves the Bogomolny equation
\begin{equation} \label{eq:bog}
d_A\Phi=*F_A
\end{equation} 
and satisfies a boundary condition which we 
do not need to consider in detail but which has two importance consequences:
\begin{itemize}
\item There is a limit $\displaystyle\lim_{r\rightarrow\infty}||\Phi||=p$, known as
the {\em mass} of the monopole.  
\item There is a well-defined map on the two-sphere at infinity $\Phi_{\infty}:S^2\rightarrow S^2$ whose 
topological degree is called the {\em charge} of the monopole. 
\end{itemize}

The Euclidean and hyperbolic metrics appear in  the Bogomolny equation \eqref{eq:bog} via the Hodge star operator $*$.  Many features of Euclidean and hyperbolic monopoles behave in a similar fashion.  However the mass is quite different. 
In the Euclidean case we can transform any monopole of non-zero mass to any other of non-zero mass by 
a suitable rescaling of space and the fields.  In the hyperbolic case we can perform the same
procedure but after rescaling hyperbolic space by dilation the curvature of the metric changes.  So
if we fix a standard copy of $\HH^3$ of scalar curvature $-1$ the monopoles of different masses
must be considered as different.  

We will be concerned primarily with the twistor picture of hyperbolic monopoles for which we need the mini-twistor
space \cite{Hit} of all oriented geodesics of $\HH^3$. Recall that a geodesic  in $\HH^3$ is determined by the 
two points where it meets the sphere at infinity. Thus the mini-twistor space is the complex manifold
\[ Q=\PP^1\times\PP^1-\bar{\Delta}\] where the point
$(\eta,\zeta)\in\PP^1\times\PP^1$ represents the geodesic that runs from
$\hat{\eta}=-1/\bar{\eta}$, the antipodal point of $\eta$, to $\zeta$ considered as
points on the sphere at infinity.  The antidiagonal $\bar{\Delta}$ has
been removed as it 
represents a geodesic from $z$ to itself.

The monopole $(A, \Phi)$ (up to gauge transformation) is determined by an algebraic curve $S \subset Q$ called the {\em spectral curve}
of the monopole. See \cite{Hit} for the Euclidean case and Atiyah's foundational work \cite{Ati} for the hyperbolic case of half-integral mass. 
The spectral curve is the collection of all oriented geodesics $\gamma$ along which the equation 
$$
(\nabla_\gamma - i \Phi ) s = 0
$$
has $L^2$ solutions.   The spectral curve of a hyperbolic monopole of charge $N$ and mass $p$ satisfies a number of conditions:

\begin{itemize}
\item It is compact of bi-degree $(N, N)$.  Hence determined by a holomorphic section of the bundle $\cO(N, N)$.  
\item It is {\em real}, that is fixed by the anti-holomomorphic involution $\tau$ which reverses the orientation 
of the geodesic.
\item There is a non-vanishing holomorphic section of $L^{2p +N}$ over $S$ where $L = \cO(1, -1)$. 
\end{itemize}

The moduli space of charge $N$ monopoles $(A, \Phi)$  modulo gauge transformations is known to be a manifold of dimension $4N - 1$
independent of the mass of the monopole.  The structure of this moduli space was proven by Atiyah \cite{Ati84} to be diffeomorphic to the space of 
based rational maps from $\PP_1$ to $\PP_1$.   Atiyah proved this in  \cite{Ati84} as the $S^1$ invariant case of the more general theorem that instantons with structure group $G$ on $S^4$ are in one-to-one correspondence with based holomorphic maps from $\PP_1$ to the loop group $\Omega G$.  Circle invariant instantons, equivalently half-integer mass hyperbolic monopoles, correspond to maps from $\PP_1$ to conjugacy classes of homomorphisms from $S^1$ to the $G$ inside $\Omega G$ which are homogeneous spaces $G/P$ of $G$.   For $SU(2)$ instantons, the homogeneous space is $SU(2)/U(1)=\PP_1$, and the moduli space of charge $N$, mass $m\in\frac12\ZZ$ hyperbolic monopoles modulo gauge transformations that are the identity at $\infty$ is homeomorphic to degree $N$ based maps from $\PP_1$ to $\PP_1$.  This space has dimension $4N$, and when we allow all gauge transformations, the rational map is well-defined up to rescaling by a factor in $U(1)$ and the dimension is $4N-1$.
In both the Euclidean and hyperbolic cases it is difficult to write down monopoles $(A, \Phi)$ or their spectral data or rational map except 
in highly symmetric cases. 

\section{JNR monopoles}

\begin{definition} 
{\sl JNR data} consists of a collection of $N+1$ positive numbers $\lambda_i > 0$ called {\sl weights} and $N+1$ distinct complex numbers  $\gamma_i$ called {\em poles}.
\end{definition}
As explained in \cite{ManSut} each collection of JNR data determines a hyperbolic monopole of mass $\frac{1}{2}$ and charge $N$. If 
we rescale the weights we obtain a gauge equivalent monopole. 
We  might expect the JNR monopoles to be a submanifold of dimension $3N+2 $ inside the $4 N - 1 $ dimensional 
moduli space of all monopoles of charge $N$.  As discussed in \cite[Section 2]{BolCocSut} this is true except for some low $N$
with the result that for $N = 1, 2, 3$ all monopoles are obtained as JNR monopoles.  There seems to be no clear geometric picture of which 
monopoles are JNR monopoles for $N > 3$. 

\subsection{JNR spectral curve}
It is shown in \cite{BolCocSut} that the spectral curve of a JNR monopole is given in terms of its JNR data $\{\lambda_i,\gamma_i\mid i=0,...,N\}$ by the equation
\begin{equation}
\label{eq:jnr-spectral}
0 = p(\eta, \zeta) = \sum_{i=0}^N  \lambda_i^2  \prod_{\substack{j= 0 \\ j \neq i}}^N  (\zeta - \gamma_j ) ( 1 + \eta \bar \gamma_j) .
\end{equation}
 Notice that, as expected from the remarks above, rescaling the weights leaves the spectral curve unchanged.

It is simple to check the basic conditions satisfied by a JNR spectral curve.  First notice that 
\begin{equation}
\label{eq:jnr-norm1}
p(\eta, \zeta) = (-1)^N \eta^N \zeta^N \overline{p(\frac{-1}{\bar \zeta}, \frac{-1}{\bar \eta})} 
\end{equation}
which shows that the spectral curve is real. Moreover
\begin{align*}
 p( \frac{-1}{\bar \zeta}, \zeta) &=  \sum_{i=0}^N  \lambda_i^2  \prod_{\substack{j= 0 \\ j \neq i}}^N  (\zeta - \gamma_j ) ( 1 +  \frac{-1}{\bar \zeta} \bar \gamma_j) \\
  &=  \frac{1}{\bar \zeta^N} \sum_{i=0}^N  \lambda_i^2  \prod_{\substack{j= 0 \\ j \neq i}}^N  |\zeta - \gamma_j |^2\\
  & \neq 0
 \end{align*}
for any $\zeta$ so the spectral curve does not intersect the anti-diagonal. Hence it is a compact subset of $Q$.

If $S$ is a spectral curve of a mass $\frac{1}{2}$ monopole it must have a non-vanishing
section $s$ of 
$$
L^{2\frac{1}{2} + N} = \cO(N+1 , -N - 1).
$$
We construct this by restricting the global meromorphic section
\begin{equation}
\label{eq:section}
s(\eta, \zeta) = \frac{ \prod_{i=0}^N (\zeta - \gamma_i ) }{ \prod_{i=0}^N (1 + \eta \bar\gamma_i ) } 
\end{equation}
of $\cO(N+1 , -N - 1)$.  

\begin{proposition}
The section $s$ defined in equation \eqref{eq:section} is non-vanishing on the spectral curve arising from JNR data.
\end{proposition}
\begin{proof}
The only points $(\eta, \zeta)$ at which this might not be true are those points  where either 
of  $\zeta = \gamma_i$ or  $\eta = \frac{-1}{\bar \gamma_j}$ holds.  The only points on the 
spectral curve like that are $(\eta, \zeta) =  (\frac{-1}{\bar \gamma_j}, \gamma_i )$ with $i \neq j$. 
Fix such an $i $ and $j$.  Then on a small enough open set around this point we have 
 $$
s(\eta, \zeta) = \sigma(\eta, \zeta) \frac{ (\zeta - \gamma_i ) }{(1 + \eta \bar\gamma_j ) } 
$$
with $\sigma$ non-vanishing.  Likewise we have 
$$
p(\eta, \zeta) = \a(\eta, \zeta)( \zeta - \gamma_i) + \b(\eta, \zeta)(1 + \eta \bar \gamma_j) + \c(\eta, \zeta)( \zeta - \gamma_i)(1 + \eta \bar \gamma_j)
$$
with $\a, \b, \c$ non-vanishing.   Hence 
$$
s(\zeta, \eta)  = \frac{-\sigma}{\a} \left( \b  + \c (\zeta - \gamma_i) \right) + p \left( \frac{1}{ \a ( 1 + \eta \bar \gamma_j ) } \right)
$$
and 
$$
\frac{-\sigma}{\a} \left( \b  + \c (\zeta - \gamma_i) \right) 
$$
is non-vanishing in a neighbourhood of $(\frac{-1}{\bar \gamma_j}, \gamma_i )$.
\end{proof}

\subsection{Scattering}
\label{sec:scattering}
Each point in $\HH^3$ defines a two-sphere
of oriented geodesics through it.  As that point approaches a point $\gamma$ at infinity this two-sphere 
of oriented geodesics approaches a pair of generators in $Q$ corresponding to the  geodesics beginning at $-1/\bar \eta = \gamma$ or ending at $\zeta = \gamma$ or the solutions to 
$$
0 = (\zeta - \gamma ) ( 1 + \eta \bar \gamma) .
$$
Notice that we can normalise $ \sum_{i=0}^N  \lambda_i^2  =1 $. Then if we let one value say $\lambda_k \to 1$ and all the others go to zero the spectral 
curve \eqref{eq:jnr-spectral} approaches the solutions of
$$
0 =   \prod_{\substack{j= 0 \\ j \neq k}}^N  (\zeta - \gamma_j ) ( 1 + \eta \bar \gamma_j) .
$$
 that correspond to the set of geodesics through the points $\gamma_0, \dots, \gamma_{k-1}, \gamma_{k+1}, \gamma_N$ at infinity in hyperbolic space.
 We can think of the monopole as approaching a monopole located at $N-1$ points near the $\gamma_0, \dots, \gamma_{k-1}, \gamma_{k+1}, \gamma_N$.
We comment again about this in Section \ref{sec:energy}.

\subsection{JNR rational map}  \label{sec:JNR-rat-map}
A remarkable result about both Euclidean and hyperbolic monopoles is that the {\em framed} moduli spaces of charge $N$ monopoles are diffeomorphic
to the space of all based rational (holomorphic) maps $\cR \colon \PP_1 \to \PP_1$.   Here based means that $\cR(\infty) = 0$.   The framed moduli space uses gauge transformations that are the identity at a chosen point on the sphere $S^2_\infty$ of directions at infinity in $\RR^2$ or $\HH^3$.  The usual moduli space is a quotient of the framed moduli space by $S^1$.    
In Cockburn's PhD thesis \cite{Coc} (see also \cite{BolCocSut}) the rational map of a JNR monopole is calculated explicitly in terms of
JNR data
as
\begin{equation}
\label{eq:JNR-rat-map}
\cR(z) = \frac{ \sum\limits_{i=0}^N \sum\limits_{j = i+1}^N \lambda_i^2 \lambda_j^2 (\gamma_i - \gamma_j)^2 \prod\limits_{\substack{k=0 \\ k \neq i, j}}^N (z - \gamma_k) }
{ \left( \sum\limits_{i=0}^N \lambda_i^2 \right) \left( \sum\limits_{j=0}^N \lambda_j^2 \prod\limits_{\substack{k=0 \\ k \neq j}}^N (z - \gamma_k)\right) }
\end{equation}
Again we note that this is unchanged if we rescale the weights. 

 We show that \eqref{eq:JNR-rat-map} can also 
be calculated by using the scattering  construction of Atiyah \cite[page 30]{Ati} (given in more detail in the Euclidean case in  \cite[page 127]{AtiHit}).

\begin{proposition} 
The rational map \eqref{eq:JNR-rat-map} for JNR monopoles agrees with the rational map Atiyah defines for hyperbolic monopoles via scattering.
\end{proposition}
\begin{proof}
 To simplify the calculation let $\sum_{i=0}^N \lambda_i^2= 1$. Then following Atiyah, the denominator of the rational map is
$$
Q(z) =  p(0, z) = \sum_{i = 0}^N \lambda_i^2 \prod_{\substack{j= 0 \\ j \neq i}}^N  (z - \gamma_j ).
$$
The numerator  $P(z)$ is the polynomial of degree $N-1$ satisfying
\begin{equation}
\label{eq:p-scattering}
P(z) = s(0, z) + (a z + b) Q(z)
\end{equation}
where $s$ is the section of $L^{m+1}$ defined in \eqref{eq:section}.
The scattering  rational map is $P(z)/Q(z)$.  We want to show that this 
agrees with $\cR(z)$.  Notice that $Q(z) $ is the denominator of \eqref{eq:JNR-rat-map}.

First we calculate that 
$$
P(z) = \prod_{i=0}^N (z - \gamma_i)^N + (a z + b ) \sum_{i = 0}^N \lambda_i^2 \prod_{\substack{j= 0 \\ j \neq i}}^N  (z - \gamma_j )
$$
and choosing $a$ and $b$ to remove the terms of order $z^{N+1}$ and $z^N$ we find that  $a = -1$ and $b =  \sum_{i = 0}^N \lambda_i^2 \gamma_i$.
Substituting this back in we obtain 

\begin{align*}
P(z) &=   \prod_{i=0}^N (z- \gamma_i )  + (-z + \sum_{k = 0}^N \lambda_k^2 \gamma_k) \sum_{i = 0}^N \lambda_i^2 \prod_{\substack{j= 0 \\ j \neq i}}^N  (z - \gamma_j ) \\
& = \sum_{i, k=0}^N \lambda_i^2 
\lambda_k^2  \left( \prod_{j=0}^N (z- \gamma_j )  - (z - \lambda_k) \prod_{\substack{j= 0 \\ j \neq i}}^N  (z - \gamma_j ) \right)\\
& = \sum_{i, k=0}^N \lambda_i^2 
\lambda_k^2  \left( 1  - \frac{(z - \lambda_k) }{(z - \gamma_i)}   \right)  \prod_{j=0}^N (z - \gamma_j)\\
& = \sum_{i, k=0}^N \lambda_i^2 \lambda_k^2  \left(  
\frac{(\lambda_k - \lambda_i) }{(z - \gamma_i)}   \right)      \prod_{j=0}^N (z - \gamma_j)\\
& = \sum_{i < k=0}^N \lambda_i^2 \lambda_k^2  \left(  
\frac{(\lambda_k - \lambda_i) }{(z - \gamma_i)}  + \frac{(\lambda_i - \lambda_k) }{(z - \gamma_k)}   \right)      \prod_{j=0}^N (z - \gamma_j)\\
& = \sum_{i < k=0}^N \lambda_i^2 \lambda_k^2  \left(  
\frac{(\lambda_k - \lambda_i)^2 }{(z - \gamma_i)(z - \gamma_k)}     \right)       \prod_{j=0}^N (z - \gamma_j)\\
\end{align*}
Hence $P(z)$ is the numerator of \eqref{eq:JNR-rat-map}
 and thus   $\cR(z) = P(z)/Q(z)$. \end{proof}

\section{A criterion to be a JNR spectral curve}
\label{sec:grid}

In this section we give a criterion for when a spectral curve must be the JNR curve of some JNR data. 

\subsection{$(N, N)$ spectral curves} 
We need some normalisations for $(N, N)$ spectral curves.  Let $S$ be a real curve not intersecting the anti-diagonal defined
by $p(\eta, \zeta) = 0$.  Then  there is a $\lambda \neq 0$ such that 
$$
p(\eta, \zeta) =  \lambda \eta^N \zeta^N \overline{p(\frac{-1}{\bar \zeta}, \frac{-1}{\bar \eta})} 
$$
Moreover if we replace $q$ by $\mu q$ for some complex number $\mu$ then $\lambda$ is replaced by $\lambda \frac{\mu}{\bar \mu}$. 
We choose $\mu$ so that 
$$
p(\eta, \zeta) =  (-1)^N\eta^N \zeta^N \overline{p(\frac{-1}{\bar \zeta}, \frac{-1}{\bar \eta})} .
$$
 
It follows from this normalisation and the fact that $S$ does not intersect the anti-diagonal  that 
 $$
 p( \frac{-1}{\bar \zeta}, \zeta) \bar \zeta^N
$$
is either negative or positive and we can, without loss of generality assume that it is positive. 
So in summary we normalise our spectral curve defining polynomials so that 
\begin{equation}
\label{eq:norm}
p(\eta, \zeta) = (-1)^N \eta^N \zeta^N \overline{p(\frac{-1}{\bar \zeta}, \frac{-1}{\bar \eta})} \\
\quad\text{and}\quad  p( \frac{-1}{\bar \zeta}, \zeta) \bar \zeta^N  > 0 .
\end{equation}
It is straightforward to check that the  JNR polynomial defined in equation \eqref{eq:jnr-spectral} is normalised this way.

\subsection{Grids}

\begin{definition}
Let us define a {\em grid} in $Q \subset \PP_1 \times \PP_1$ to be a subset of the form 
$$
\cG = \{ (\frac{-1}{\bar \gamma_i}, \gamma_j ) \mid 0 \leq i \neq j \leq N \}
$$
for $\gamma_i$ a collection of $N+1$ distinct points in $\PP_1$.
\end{definition}

  Notice that a grid is a collection 
of $N(N+1)$ points which is real and does not intersect the anti-diagonal.  Sometimes we need to 
include the points we have excluded to avoid the anti-diagonal so we define $\bar \cG$ to be 
$$
\bar \cG = \{ (\frac{-1}{\bar \gamma_i}, \gamma_j ) \mid 0 \leq i ,  j \leq N \}
$$
which is $(N+1)^2$ points. 
We say that a curve $S$ in $\PP_1 \times \PP_1$ {\sl admits a grid} if there is a grid $\cG$  with $\cG \subset S$.

\subsection{JNR monopoles and grids}

Clearly a JNR spectral curve admits the grid determined by the poles in the JNR data.  In fact the converse is true as shown by the following proposition.
\begin{proposition} 
Let $S$ be a real curve of degree $(N, N)$ not intersecting the anti-diagonal and  admitting a grid $\cG$ determined by 
points $\gamma_0, \dots, \gamma_N$. Then 
there exist  positive real numbers $\lambda_0, \dots, \lambda_N$ such that $S$ is the spectral 
curve determined by the JNR data with  poles $\gamma_0, \dots, \gamma_N$ and weights $\lambda_0, \dots, \lambda_N$.
\end{proposition}
\begin{proof}
Let $q$ define $S$ and normalise it as above. Define
 $$
 \lambda_i^2 =  \bar \gamma_i^N \frac{ q( \frac{-1}{\bar \gamma_i}, \gamma_i )}  { \prod\limits_{\substack{j= 0 \\ j \neq i}}^N |\gamma_i - \gamma_j |^2 }
 $$
and use it to define $p$ by equation \eqref{eq:jnr-spectral}. It follows that $q$ and $p$ agree on $\bar \cG$.  So from the exact sequence in sheaf cohomology 
$$
0 \to \cO_Q(-1, -1) \to \cO_Q(N, N) \to \cO_{\bar\cG}(N, N) \to 0 
$$
we deduce that $p = q$.
\end{proof}

Hence we have:
\begin{corollary}  \label{JNRgrid}
A spectral curve comes from JNR data if and only if it admits a grid.
\end{corollary}

Consider the formal dimension count for grids and JNR data.  Notice that the vector space $H^0(Q, \cO(N, N)) $ has dimension $(N+1)^2$ so the space of curves has dimension $(N+1)^2 - 1 = N^2 + 2N$
real dimensions.   The closed grid has $(N+1)^2 = N^2 + 2N + 1$ points.  If we remove the intersection with the anti-diagonal
that leaves the grid which has $N^2 + N$ points in real pairs.  If we require the real curve to go through the grid that is $\frac{1}{2} (N^2 + N)$ constraints each contributing two real dimensions.  So we might reasonably expect the 
space of spectral curves admitting grids to be determined by $N^2 + 2N - (N^2 + N) = N$ parameters. This fits with $N+1$ weights less the overall scaling factor of $1$.

\section{Holomorphic spheres}   \label{sec:holsph}

In \cite{MurNorSin} with Singer we constructed a new invariant, particular to hyperbolic monopoles called the holomorphic sphere.
A hyperbolic monopole has a well-defined limit at infinity given by a
reducible connection $A_{\infty}$ over a two-sphere.  Denote by
$F_{A_{\infty}}$ the curvature of the reducible connection.  We showed that:
\begin{theorem} \cite[Theorem 1]{MurNorSin}  \label{th:sphere}
    An $SU(2)$ hyperbolic monopole $(A,\Phi)$ of charge $N$ is determined by
    a degree $N$ holomorphic embedding $q:\PP^1\rightarrow\PP^N$
    with the properties:
\begin{itemize}      
\item[(i)] $\Sigma=\{(\eta, \zeta)\in\PP^1\times\PP^1\ |\ \langle
    q(\hat{\eta}),q(\zeta)\rangle=0\}$ is the spectral curve of
    $(A,\Phi)$;\\
    
\item[(ii)] $q$ is uniquely defined up to the action of $U(N+1)$ on its image;\\
    
\item[(iii)] $F_{A_{\infty}}=q^*\omega$, for $\omega$ the
    Kahler form on $\PP_N$. 
\end{itemize}    
\end{theorem} 

Here we have changed the notation using $(\eta, \zeta)$ for $(w, z)$ in \cite{MurNorSin}. Also $\<\ , \ \>$ is the Hermitian 
inner product, conjugate in the first factor.

Property (iii) of Theorem~\ref{th:sphere} can be used to prove that $F_{A_{\infty}}$ uniquely determines the monopole $(A,\Phi)$ up to gauge transformation.  It potentially gives one a way to study monopoles from a two-dimensional perspective much like the AdS-CFT correspondence which proposes a relationship between string theory on anti-de Sitter spacetime and conformal field theory on the boundary.  In general a degree $N$ holomorphic embedding $q:\PP^1\rightarrow\PP^N$ does not correspond to a hyperbolic monopole.  The first challenge then, to  implement such ideas, is to understand the space of  holomorphic spheres, or fields at infinity $F_{A_{\infty}}$.   This is achieved for JNR monopoles in the proposition below.

\begin{proposition}
\label{prop:holo-sphere}
The holomorphic sphere for a JNR monopole is given by 
\begin{equation}  \label{holsph}
q(z) = \left[ \frac{\lambda_0}{z - \gamma_0} , \frac{\lambda_1}{z - \gamma_1} , \cdots, \frac{\lambda_N}{z - \gamma_N}\right].
\end{equation}
\end{proposition} 
\begin{proof}
We need to connect with the  construction of the holomorphic sphere in \cite{MurNorSin}. Comparing the normalisations in \cite[Lemma 2.1]{MurNorSin} with those in equation \eqref{eq:norm} we see that
$\psi$ in \cite{MurNorSin} is related to $p$ by 
$$
\psi(\eta, \zeta) = \frac{ (-1)^N}{\eta^N} p (\eta, \zeta).
$$
Then we need to use  \cite[Theorem 4]{MurNorSin} which shows that  $q$ is determined by  $\psi$ and satisfies
\begin{equation}
\label{eq:sphere-q}
\psi(\eta, \zeta) = \langle q( \frac{-1}{\bar \eta}), q( \zeta) \rangle .
\end{equation}
In fact from the discussion there it follows that if $q \colon \PP_1 \to \PP_N$ is of degree $N$ and full, that is does not have its image in a proper
subspace, then equation \eqref{eq:sphere-q} determines $q$ up to the action of $U(N+1)$.  So we have
\begin{align*}
\psi(\eta, \zeta)  &= (\frac{-1}{\eta})^N \sum_{i=0}^N  \lambda_i^2  \prod_{\substack{j= 0 \\ j \neq i}}^N  (\zeta - \gamma_j ) ( 1 + \eta \bar \gamma_j) \\
  &= \sum_{i=0}^N  \lambda_i^2  \prod_{\substack{j= 0 \\ j \neq i}}^N  (\zeta - \gamma_j ) ( \frac{-1}{\eta} - \bar \gamma_j) \\
              & =   \sum_{i=0}^N \left(\frac{\lambda_i }{\frac{-1}{\eta} -  \bar \gamma_i}  \prod_{j= 0}^N ( \frac{-1}{\eta} - \bar \gamma_j) \right) \left( \frac{\lambda_i}{z - \gamma_i }  \prod_{j= 0}^N  (\zeta - \gamma_j ) \right)\\ 
              & =   \sum_{i=0}^N \overline{\left(\frac{\lambda_i }{\frac{-1}{\bar \eta} -  \bar \gamma_i}  \prod_{j= 0}^N ( \frac{-1}{\bar \eta} - \bar \gamma_j) \right)} \left( \frac{\lambda_i}{z - \gamma_i }  \prod_{j= 0}^N  (\zeta - \gamma_j ) \right)\\
              &= \langle q( \frac{-1}{\bar \eta}), q(\zeta) \rangle
\end{align*}
so that 
\begin{align}
\label{eq:holo-sphere}
q(z) &= \left[  \frac{\lambda_0}{z - \gamma_0} \left(\prod_{j= 0}^N  (z - \gamma_j )\right) ,  \cdots,   \frac{\lambda_N}{z - \gamma_N} \left(\prod_{j= 0}^N  (z - \gamma_j )\right)\right]  \\
&= \left[ \frac{\lambda_0}{z - \gamma_0} ,  \cdots, \frac{\lambda_N}{z - \gamma_N} \right] \nonumber. 
\end{align}
has degree $N$ and satisfies \eqref{eq:sphere-q}.  If $q$ is not full then there is some $(\alpha_0, \dots, \alpha_N)$ such that 
$$
\sum_{k=0}^N  \frac{\alpha_k \lambda_k }{(z - \gamma_k)} = 0
$$
which is not possible as  the $\gamma_k $ are distinct and  the $\lambda_k >0$.  Hence the result follows.
\end{proof}
 We shall see in Section \ref{sec:energy} that along with $F_{A_{\infty}}=q^*\omega$
this can be used to give a formula for the energy density of a JNR monopole at infinity.
 
\subsection{$SO(3)$ action on JNR monopoles.}
A consequence of Corollary~\ref{JNRgrid} and Proposition~\ref{prop:holo-sphere} is that the action by isometries of $SO(3)$ on $\HH^3$ preserves JNR monopoles.  To see this, note that the spectral curve is rotated by the ambient $SO(3)$ action on $\PP^1\times\PP^1$.  Hence the action of $SO(3)$ on the holomorphic sphere is given by $q(z)\mapsto q(g^{-1}\cdot z)$.  (Note that this only works for $SO(3)$ and fails for $PSL(2,\CC)$.)  Here $g^{-1}\cdot z=\frac{\bar{a}z-b}{\bar{b}z+a}$ for $g\in SU(2)/\pm I\cong SO(3)$.  Hence the $i$th term of $q$ becomes
$$\frac{\lambda_i}{z - \gamma_i}\mapsto\frac{\lambda_i}{\frac{\bar{a}z-b}{\bar{b}z+a} - \gamma_i}=\frac{\lambda_i(\bar{b}z+a)}{(\bar{a}-\bar{b}\gamma_i)z-(a\gamma_i+b)}=\frac{\frac{\lambda_i}{\bar{a}-\bar{b}\gamma_i}(\bar{b}z+a)}{z-\frac{a\gamma_i+b}{\bar{a}-\bar{b}\gamma_i}}.
$$
After dividing by the common factor of $\bar{b}z+a$ and replacing $\bar{a}-\bar{b}\gamma_i$ by $|\bar{a}-\bar{b}\gamma_i|$ via a diagonal $U(N+1)$ action on $\PP^N$ we get
$$q(z) = \left[ \frac{\lambda_0}{z - \gamma_0} , \frac{\lambda_1}{z - \gamma_1} , \cdots, \frac{\lambda_N}{z - \gamma_N} \right]\mapsto
\left[ \frac{\frac{\lambda_0}{|\bar{a}-\bar{b}\gamma_0|}}{z - \frac{a\gamma_0+b}{\bar{a}-\bar{b}\gamma_0}} , \frac{\frac{\lambda_1}{|\bar{a}-\bar{b}\gamma_1|}}{z - \frac{a\gamma_1+b}{\bar{a}-\bar{b}\gamma_1}} , \cdots, \frac{\frac{\lambda_N}{|\bar{a}-\bar{b}\gamma_N|}}{z - \frac{a\gamma_N+b}{\bar{a}-\bar{b}\gamma_N}} \right].
$$
Since the holomorphic sphere uniquely determines the monopole, we conclude that 
the $SO(3)$ action on JNR data is given by
$$\left(\begin{array}{cc}a&b\\-\bar{b}&\bar{a}\end{array}\right)\cdot \{\lambda_j,\gamma_j\}=\left\{\frac{\lambda_j}{|\bar{a}-\bar{b}\gamma_j|},\frac{a\gamma_j+b}{\bar{a}-\bar{b}\gamma_j}\right\}.
$$
In particular the $\frac12(N+1)(N+2)$ functions
$$
\frac{\lambda_i^2\lambda_j^2}{|1+\overline{\gamma}_i\gamma_j|^2}
$$
are invariant functions.  The orbit of $SO(3)$ is generically 3-dimensional for $N>1$ hence there are at least $\frac12(N+1)(N+2)+3-(3N+3)=\frac12(N-1)(N-2)$ relations among the invariant functions

\section{Projection to rational maps.}  \label{sec:sphrat}

The holomorphic sphere $q:\PP^1\rightarrow\PP^N$ uniquely determines the monopole $(A,\Phi)$ up to gauge transformation.  The rational map is defined by choosing a point on
the two-sphere at infinity $p\in S^2_\infty$, restricting gauge transformations to limit to the identity at $p$, i.e. framed gauge transformations, then scattering from $p$.  As described in Section~\ref{sec:JNR-rat-map}, the full group of gauge transformations reduces to $U(1)$ in the limit at the point $p$, so if we do not restrict to based gauge transformations at $p$, the rational map is well-defined up to rescaling by $e^{i\theta}$.

In \cite{MurNorSin}  we show in  Theorem 3 that, for each monopole, there is a  linear map  $\PP_N \to L_p = \PP_1 \subset \PP_N$ which composed with the 
holomorphic sphere $q$ gives the rational map up to rescaling. The point $p\in S^2_\infty$ is the scattering direction which is chosen to be
$\infty$ in Section~\ref{sec:JNR-rat-map}, and takes the form \eqref{eq:JNR-rat-map}.

\begin{lemma} 
There exists a unique $a = (a_0,a_1,...a_N)\in\CC^{N+1}$ such that the rational map is given by
$$
\cR(z)=\frac{\frac{a_0\lambda_0}{z-\gamma_0}+\frac{a_1\lambda_1}{z-\gamma_1}+...+\frac{a_N\lambda_N}{z-\gamma_N}}
{\frac{\lambda_0^2}{z-\gamma_0}+\frac{\lambda_1^2}{z-\gamma_1}+...+\frac{\lambda_N^2}{z-\gamma_N}} = \frac{ \langle \bar a , q(z) \rangle}{ \langle 
q(\infty), q(z) \rangle }
$$
\end{lemma}
\begin{proof}
To simplify the calculation assume that $\sum_{i = 0}^N \lambda_i^2 = 1$.  This just changes the 
definition of the $a_i$.  
If we divide through by $\prod_{k=0}^N (z - \gamma_k)$ we get the correct denominator and the numerator becomes
\begin{align*} 
\sum\limits_{i=0}^N \sum\limits_{j = i+1}^N  \lambda_i^2 \lambda_j^2 \frac{(\gamma_i - \gamma_j)^2}{(z - \gamma_i)(z - \gamma_j)} 
&= \frac{1}{2} \sum\limits_{i, j=0}^N  \lambda_i^2 \lambda_j^2 \frac{(\gamma_i - \gamma_j)^2}{(z - \gamma_i)(z - \gamma_j)} \\
&=\frac{1}{2} \sum\limits_{i, j=0}^N   \lambda_i^2 \lambda_j^2 \left(\frac{ \gamma_i - \gamma_j} {z - \gamma_i} + \frac{ \gamma_j - \gamma_i}{z-\gamma_j} \right)\\
&= \sum\limits_{i, j=0}^N   \lambda_i^2 \lambda_j^2 \frac{ \gamma_i - \gamma_j} {z - \gamma_i} \\
&= \sum\limits_{i=0}^N  \left(\sum_{j=0}^N \lambda_i \lambda_j^2 (\gamma_i - \gamma_j) \right) \frac{ \lambda_i} {z - \gamma_i}.
\end{align*}
This gives the required result with
$$
a_i = \sum_{j=0}^N \lambda_i \lambda_j^2 (\gamma_i - \gamma_j).
$$
Note that $a_0\lambda_0+a_1\lambda_1+...+a_N\lambda_N=0$ so that $\cR(\infty)=0$.
\end{proof}

Hence
$$\cR(z)=\frac{\frac{\gamma_0\lambda_0^2}{z-\gamma_0}+\frac{\gamma_1\lambda_1^2}{z-\gamma_1}+...+\frac{\gamma_N\lambda_N^2}{z-\gamma_N}}
{\frac{\lambda_0^2}{z-\gamma_0}+\frac{\lambda_1^2}{z-\gamma_1}+...+\frac{\lambda_N^2}{z-\gamma_N}}-\langle\gamma\rangle
$$
where
$$\langle\gamma\rangle=\frac{\sum_{j=0}^N\lambda_j^2\gamma_j}{\sum_{j=0}^N\lambda_j^2}.
$$
Thus $\cR(z)$ is achieved up to scale by projecting $q(z)$ onto the line $L\subset\PP_N$ spanned by
$$[\lambda_0,...,\lambda_N],\quad [\bar{\gamma}_0\lambda_0,...,\bar{\gamma}_N\lambda_N]
$$
which uses Gram-Schmidt as follows.  Set $v_1=(\lambda_0,...,\lambda_N)$.  Then $$v_2=(\bar{\gamma}_0\lambda_0,...,\bar{\gamma}_N\lambda_N)-\langle\bar{\gamma}\rangle v_1$$ 
is orthogonal to $v_1$ and projection is given by
$$ Pq(z)=\frac{\langle v_1,q(z)\rangle v_1}{\langle v_1,v_1\rangle}+\frac{\langle v_2,q(z)\rangle v_2}{\langle v_2,v_2\rangle}
$$
which is essentially $\cR(z)$ except for a scale factor of 
$$ \frac{\langle v_2,v_2\rangle}{\langle v_1,v_1\rangle}.
$$

Instead of scattering from $\infty\in S^2_\infty$, as one might expect, the rational map obrtained by scattering from $\gamma_i$ is simpler.  It is of the form:
$$ 
\cR_{\gamma_i}(z)=(z-\gamma_i)\sum_{j\neq i}\frac{a_{ij}\lambda_i}{z-\gamma_j}
$$
for some $a_{ij}\in\CC$ which is obtained by projection of $q$ onto $L_\zeta = \PP_1 \subset \PP_N$.

\section{The energy density at infinity}
\label{sec:energy}
From \cite{MurNorSin} we know that the energy density of the monopole at infinity is given by $q^*\omega$ where $\omega$ 
is the K\"ahler form on $\PP_n$ and $q \colon \PP_1 \to \PP_N$ the holomorphic sphere.  There are various ways to calculate this and we will use 
\begin{equation}
\label{eq:E-def}
E(q) = {\vert q \wedge \d q\vert^2 \over \vert q\vert^4}.
\end{equation}
Here we have lifted $q$ to a map into $\CC^{n+1}$ and we regard $q \wedge \d q$ as an element of $\bigwedge^2 \CC^{n+1}$ with 
the inner product induced from $\CC^{n+1}$.

If we let $e^0, \dots, e^N$ be the standard basis of $\CC^{N+1}$ then from \eqref{eq:holo-sphere} we can take
$$
q(z) = \sum_{i=0}^N  \frac{\lambda_i} {z - \gamma_i} e^i.
$$
so that 
$$
\partial q(z) = \sum_{i=0}^N  \frac{-\lambda_i} {(z - \gamma_i)^2} e^i.
$$
and thus 
$$
q \wedge \partial q = \sum_{0 \leq i < j \leq N} \left( \frac{\lambda_i}{(z - \gamma_i)^2} \frac{\lambda_j}{(z - \gamma_j)} 
            - \frac{\lambda_j}{(z - \gamma_j)^2} \frac{\lambda_i}{(z - \gamma_i)} \right) e^i \wedge e^j
$$
Substituting into the formula for the energy density \eqref{eq:E-def} we obtain
\begin{align}
\label{eq:energy-density}
E(q)(z) &= \frac{\sum\limits_{0 \leq i < j \leq N} \left| \frac{\lambda_i}{(z - \gamma_i)^2} \frac{\lambda_j}{(z - \gamma_j)} 
            - \frac{\lambda_j}{(z - \gamma_j)^2} \frac{\lambda_i}{(z - \gamma_i)} \right|^2 }
            { \left| \underset{i=0}{\overset{N}{\sum}}  \left| \frac{\lambda_i} {(z - \gamma_i)} \right|^2 \right|^2 } 
            \end{align}

A calculation shows that 
\begin{equation}
\label{eq:energy-max}
E(q)(\gamma_k) = \sum\limits_{0 \leq j \neq k \leq N} \left( \frac{\lambda_j}{\lambda_k}\right)^2 \frac{1}{| \gamma_k - \gamma_j |^2 }.
\end{equation}

\begin{example}   Take $N=3$, $\lambda_0 = \lambda_1 = \lambda_2 = \lambda_3 = 1$, $\gamma_0 = 0$, $\gamma_1 = i$,  $\gamma_2 = 2$ and 
$\gamma_3 = 1 - i$.
Then we obtain the plots in \ref{fig:e-3d} and \ref{fig:e-contour}.  

If the parameters are varied by lowering all the values of $\lambda_i$ except say the $k$th one uniformly towards $0$ then 
as discussed in Section \ref{sec:scattering} we expect that the monopole (at least its spectral curve) approaches
a monopole located near the points $\gamma_0, \dots, \gamma_{k-1}, \gamma_{k+1}, \gamma_N$ at infinity. 
Similarly we see here that $E(q)(z)$ approaches a sum of delta functions at $\gamma_0, \dots, \gamma_{k-1}, \gamma_{k+1}, \gamma_N$.

\begin{figure}
\begin{center}
\includegraphics[scale=0.6]{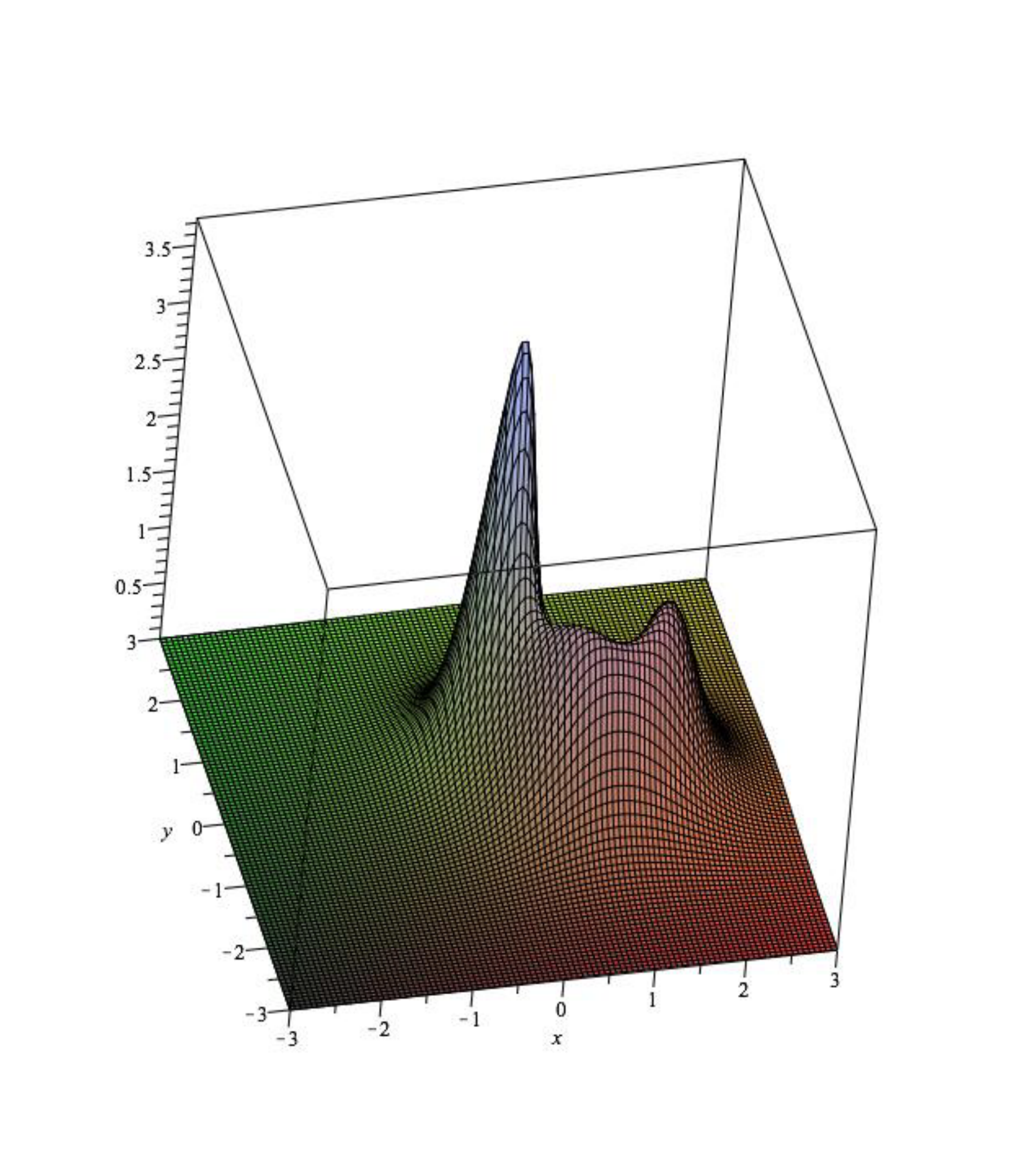}
\end{center}
\caption{Energy density at infinity for $N = 3$ JNR monopole}
\label{fig:e-3d}
\end{figure}

\begin{figure}
\begin{center}
\includegraphics[scale=0.5]{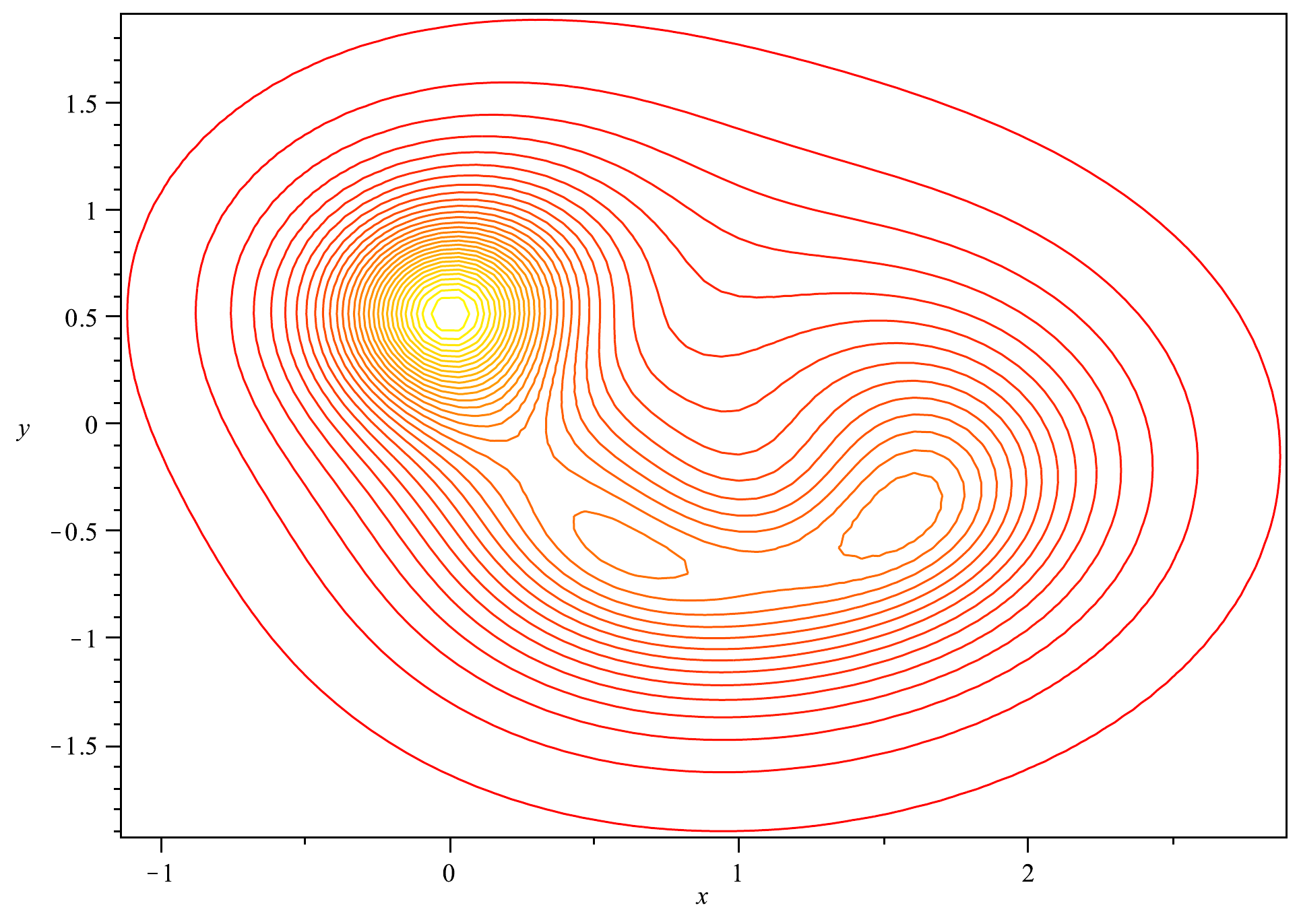}
\end{center}
\caption{Energy density at infinity for $N = 3$ JNR monopole}
\label{fig:e-contour}
\end{figure}
\end{example}

In principle it should be possible to calculate the information metric (see for example \cite{Mur})  induced from the energy density at infinity in terms of the JNR data.  


\end{document}